\documentclass[12pt]{amsart}
\usepackage{a4wide}
\usepackage[T1]{fontenc}
\usepackage{amssymb,amsmath,amsthm,latexsym}
\usepackage{mathrsfs}
\usepackage[usenames,dvipsnames]{color}
\usepackage{euscript}
\usepackage{graphicx}
\usepackage{mdwlist}
\usepackage{enumerate}
\usepackage{mathtools,dsfont,wasysym}
\usepackage{stmaryrd}
\usepackage{todonotes}
\usepackage{hyperref}
\hypersetup{colorlinks=true, linkcolor=blue, citecolor=blue, urlcolor=cyan}

\newtheorem{theorem}{Theorem}[section]
\newtheorem{lemma}[theorem]{Lemma}
\newtheorem{corollary}[theorem]{Corollary}
\newtheorem{proposition}[theorem]{Proposition}

\theoremstyle{remark}
\newtheorem{remark}[theorem]{Remark}

\theoremstyle{definition}
\newtheorem{definition}[theorem]{Definition}

\newtheorem{example}[theorem]{Example}
\numberwithin{equation}{section}
\newcounter{maintheorem}

\makeatother

\newcommand{\R}{\mathbb{R}}

\newcommand{\e}{\varepsilon}

\newcommand{\n}{\left\Vert\cdot\right\Vert}
\newcommand{\nn}[1]{{\left\vert\kern-0.25ex\left\vert\kern-0.25ex\left\vert #1 
\right\vert\kern-0.25ex\right\vert\kern-0.25ex\right\vert}}

\renewcommand{\leq}{\leqslant}
\renewcommand{\geq}{\geqslant}
\DeclareMathOperator{\suppt}{suppt}

\newcounter{smallromans}

{\end{list}}

\begin{document}
\title{Small semi-Eberlein compacta and inverse limits}

\author[C.~Correa]{Claudia Correa}
\address[C.~Correa]{Centro de Matem\'atica, Computa\c c\~ ao e Cogni\c c\~ ao, Universidade Federal do ABC, Avenida dos Estados, 5001, Santo Andr\' e, Brazil}
\email{claudiac.mat@gmail.com, claudia.correa@ufabc.edu.br}

\author[T.~Russo]{Tommaso Russo}
\address[T.~Russo]{Institute of Mathematics\\ Czech Academy of Sciences\\ \v{Z}itn\'a 25, 115 67 Prague 1\\ Czech Republic; and Department of Mathematics\\Faculty of Electrical Engineering\\Czech Technical University in Prague\\Technick\'a 2, 166 27 Prague 6\\ Czech Republic}
\email{russo@math.cas.cz, russotom@fel.cvut.cz}

\author[J.~Somaglia]{Jacopo Somaglia}
\address[J.~Somaglia]{Dipartimento di Matematica ``F. Enriques'' \\Universit\`a degli Studi di Milano\\Via Cesare Saldini 50, 20133 Milano\\Italy}
\email{jacopo.somaglia@unimi.it}

\thanks{Research of C.~Correa was supported by Funda\c{c}\~{a}o de Amparo \`a Pesquisa do Estado de S\~{a}o Paulo (FAPESP) grants 2018/09797-2 and 2019/08515-6.\\
Research of T.~Russo was supported by the GA\v{C}R project 20-22230L; RVO: 67985840 and by Gruppo Nazionale per l'Analisi Matematica, la Probabilit\`a e le loro Applicazioni (GNAMPA) of Istituto Nazionale di Alta Matematica (INdAM), Italy.\\ 
Research of J.~Somaglia was supported  by Universit\`a degli Studi di Milano, and by Gruppo Nazionale per l'Analisi Matematica, la Probabilit\`a e le loro Applicazioni (GNAMPA) of Istituto Nazionale di Alta Matematica (INdAM), Italy.}

\date{\today}

\begin{abstract}
We study properties of semi-Eberlein compacta related to inverse limits. We concentrate our investigation on an interesting subclass of small semi-Eberlein compacta whose elements are obtained as inverse limits whose bonding maps are semi-open retractions.
\end{abstract}
\maketitle

%-----------------------------------------------------------%
%                                                           %
% 						INTRODUCTION 						%
%                                                           %
%-----------------------------------------------------------%

\section{Introduction}

The notion of semi-Eberlein compact space was introduced by Kubi\'s and Leiderman in \cite{KL}, as a natural generalization of the classical notion of Eberlein compact. We say that a compact space $K$ is \emph{semi-Eberlein} if there exists a homeomorphic embedding $h:K\to\mathbb{R}^{\Gamma}$ such that $h^{-1}[c_0(\Gamma)]$ is dense in $K$, where
\[
c_0(\Gamma):=\{x\in \mathbb{R}^{\Gamma}\colon (\forall \varepsilon > 0) |\{\gamma\in \Gamma\colon |x(\gamma)|>\varepsilon\}| < \omega\}\subset \mathbb{R}^{\Gamma}.
\]
Clearly the class of semi-Eberlein compacta contains every Eberlein compact space and it is contained in the class of Valdivia compacta. It is easy to see that the generalized Cantor cube $2^{\kappa}$ is an example of a semi-Eberlein compact space that is not Eberlein and in \cite[Corollary~5.3]{KL} it was shown that the Valdivia space $[0, \omega_1]$ is not semi-Eberlein. It is worth mentioning that even though Eberlein and Valdivia compact spaces have been widely studied (see for example \cite{AL,BRW,DG,Ka,KM} and more recently \cite{CCS,CT,S}), the class of semi-Eberlein compacta has not been thoroughly investigated yet. Indeed, after its introduction in \cite{KL}, this class was studied only in the papers \cite{CCS} and \cite[Section 4.3]{HRST}.

The goal of this work is to investigate properties of semi-Eberlein compacta related to inverse limits. In \cite{KM} those properties were investigated in the context of Valdivia compact spaces. 
It was shown in \cite[Proposition~2.6]{KM} that every Valdivia compact space can be obtained as the inverse limit of a certain kind of inverse system. In Theorem \ref{semiEberlein-InverseLimit}, we present the semi-Eberlein version of this result. The key ingredient in the proof of Theorem \ref{semiEberlein-InverseLimit} is the characterization of semi-Eberlein compacta in terms of retractional skeletons presented in \cite[Theorem~B]{CCS}. Moreover in \cite{KM} a characterization of small Valdivia compact spaces via inverse limits was established as a consequence of \cite[Proposition~2.6]{KM} and \cite[Corollary~4.3]{KM}. More precisely, a small compact space is Valdivia if and only if it is the inverse limit of a continuous inverse system of compact metric spaces whose bonding maps are retractions. In Theorem \ref{semiEberlein-w1-Shrinking-InverseLimit}, we establish a version of this characterization in the context of small semi-Eberlein compacta. Recall that a topological space $K$ is said to be \emph{small} if its weight $w(K)$ is $\omega_1$.

Finally, inspired by a stability result for semi-Eberlein compact spaces presented in \cite{KL}, in Section \ref{sec:RS} we introduce a subclass of the class of small semi-Eberlein compacta. More precisely, it was shown in \cite[Theorem~4.2]{KL} that every inverse limit of a continuous inverse system of compact metric spaces whose bonding maps are semi-open retractions is semi-Eberlein. We define $\mathcal{RS}$ as the class comprising all such inverse limits\footnote{The notation for this class refers to `semi-open retractions' and it is inspired by classes $\mathcal{R}$ and $\mathcal{RC}$ considered in \cite{Ku}.}. Having in mind the aforementioned characterization of small Valdivia compacta, it was quite natural to conjecture that every small semi-Eberlein compactum belongs to $\mathcal{RS}$. Rather surprisingly, it turns out that this is not the case. For instance, in Section \ref{sec:RS} we show that $\mathcal{RS}$ does not even contain every small Eberlein compact space. More precisely, in Corollary \ref{scattered-corson}, we show that if $K$ is a nonmetrizable scattered Eberlein compact space, then $K$ does not belong to $\mathcal{RS}$. In Subsection \ref{subsec: stability-RS}, we present some stability results for the class $\mathcal{RS}$.

\section{Relations between semi-Eberlein compacta and inverse limits}\label{sec: 2}

In this section we establish relations between semi-Eberlein compacta and inverse limits. In order to do so, we use the characterization of semi-Eberlein spaces in terms of retractional skeletons presented in \cite{CCS} and explore the deep connection between retractional skeletons and inverse limits. To understand this connection, we need to recall some notions and properties of those objects. 
Here all topological spaces are assumed to be Hausdorff and the following monographs contain basic definitions and results that are used without specific reference:  \cite{E}, \cite{J}, and \cite{KKL}.
Recall that an up-directed partially ordered set $\Sigma$ is said to be \emph{$\sigma$-complete} if every countable and up-directed subset of $\Sigma$ admits supremum in $\Sigma$; equivalently, every increasing sequence in $\Sigma$ admits supremum in $\Sigma$.

\begin{definition}\label{def: r-skeleton}
A \emph{retractional skeleton} on a compact space $K$ is a family of continuous internal retractions $\mathfrak s=(R_s)_{s \in \Sigma}$ on $K$ indexed by an up-directed and $\sigma$-complete partially ordered set $\Sigma$, such that:
\begin{enumerate}
\item[(i)]\label{metric-image} $R_{s}[K]$ is a metrizable compact space, for every $s\in\Sigma$,
\item[(ii)]\label{images-increasing-kernel-decreasing} if $s,t\in\Sigma$ and $s\le t$ then $R_{s}=R_t\circ R_s=R_s\circ R_t$,
\item[(iii)] \label{skeleton-continuity}given an increasing sequence $(s_n)_{n\in\omega}$ in $\Sigma$, if $s=\sup_{n\in\omega}s_{n}\in\Sigma$, then $R_{s}(x)=\lim_{n\to \infty}R_{s_{n}}(x)$, for every $x\in K$,
\item[(iv)] \label{induced-subset-big} for every $x\in K$, $x=\lim_{s\in\Sigma}R_{s}(x)$.
\end{enumerate}
We say that $\bigcup_{s \in \Sigma}R_s[K]$ is the set \emph{induced} by the retractional skeleton $\mathfrak s$ and we denote it by $D(\mathfrak{s})$.
\end{definition}

Note that condition (iv) implies that $D(\mathfrak{s})$ is dense in $K$.

\begin{definition}\label{def:inverse-system}
An \emph{inverse system} of compact spaces indexed by an up-directed partially ordered set $\Sigma$ is a pair $\mathbb{S}=\big((K_s)_{s \in \Sigma}, (p_{s}^{t})_{s \le t}\big)$, where $K_s$ is a compact space, for every $s \in \Sigma$, $p_{s}^{t}:K_t \to K_s$ is a continuous map, for every $s, t \in \Sigma$ with $s \le t$ and the following conditions are satisfied
\begin{enumerate}
    \item[(i)] $p_{s}^{s}$ is the identity of $K_s$, for every $s \in \Sigma$;
    \item[(ii)] If $s_1 \le s_2 \le s_3$, then $p_{s_1}^{s_3}= p_{s_1}^{s_2} \circ p_{s_2}^{s_3}$.
\end{enumerate}
The maps $p_{s}^{t}$ are called \emph{bonding maps}. We say that a pair $\big(K, (p_s)_{s \in \Sigma}\big)$ is a \emph{cone} over $\mathbb{S}$ if $K$ is a compact space, $p_s: K \to K_s$ is a continuous map, for every $s \in \Sigma$ and it holds that $p_s=p_{s}^{t} \circ p_t$, for every $s \le t$. The maps $p_s$ are called \emph{projections}. An \emph{inverse limit} of $\mathbb{S}$ is a cone $\big(K, (p_s)_{s \in \Sigma}\big)$ over $\mathbb{S}$ such that given any cone $\big(L, (q_s)_{s \in \Sigma}\big)$ over $\mathbb{S}$, there exists a unique continuous function $f:L \to K$ such that $q_s=p_s \circ f$, for every $s \in \Sigma$. 
\end{definition}

Recall that every inverse system of compact spaces admits an inverse limit, which is unique up to homeomorphisms (more precisely, cone homeomorphisms). Throughout this work, we will always assume that the bonding maps are onto. This implies that the projections from the inverse limit are also onto \cite[Corollary~3.2.15]{E}. It is not hard to see that a cone $\big(K, (p_s)_{s \in \Sigma}\big)$ over an inverse system of compact spaces $\mathbb{S}$ is the inverse limit of $\mathbb{S}$ if and only if the family $\{p_s: s \in \Sigma\}$ separates the points of $K$. Indeed, this follows from the explicit representation of the inverse limit of $\mathbb{S}=\big((K_s)_{s \in \Sigma}, (p_{s}^{t})_{s \le t}\big)$ as
\[K_\mathbb{S}=\{(x_s)_{s \in \Sigma} \in \Pi_{s \in \Sigma}K_s: p_{s}^{t} (x_t)=x_s, \ \forall s \le t\},\]
equipped with the usual projections.
Note that if $\mathbb{S}=\big((K_s)_{s \in \Sigma}, (p_{s}^{t})_{s \le t}\big)$ is an inverse system and $T$ is an up-directed subset of $\Sigma$, then $\mathbb{S}|_T:=\big((K_s)_{s \in T}, (p_{s}^{t})_{s \le t}\big)$ is again an inverse system. Moreover if $T$ is a cofinal subset of $\Sigma$
and $\big(K, (p_s)_{s \in \Sigma}\big)$ is the inverse limit of $\mathbb{S}$, then $\big(K, (p_t)_{t \in T}\big)$ is the inverse limit of $\mathbb{S}|_{T}$. Here a central role is played by continuous and $\sigma$-complete inverse systems. We say that an inverse system $\mathbb{S}=\big((K_s)_{s \in \Sigma}, (p_{s}^{t})_{s \le t}\big)$ is \emph{continuous} if for every up-directed subset $T$ of $\Sigma$ that admits supremum $t=\sup T$ in $\Sigma$, it holds that $\big(K_t, (p_{s}^{t})_{s \in T}\big)$ is the inverse limit of $\mathbb{S}|_T$. We say that $\mathbb{S}$ is \emph{$\sigma$-complete} if $\Sigma$ is $\sigma$-complete and for every countable and up-directed subset $T$ of $\Sigma$ with $t=\sup T$ in $\Sigma$, it holds that $\big(K_t, (p_{s}^{t})_{s \in T}\big)$ is the inverse limit of $\mathbb{S}|_T$.

The theory developed in \cite{CCS} allows us to show in Lemma \ref{inverse-system-canonical-retractions} that every compact space that admits a retractional skeleton is the inverse limit of a continuous inverse system such that the bonding maps and the projections are retractions with nice properties. The key ingredient of Lemma \ref{inverse-system-canonical-retractions} is the notion of canonical retractions associated to suitable models that was introduced in \cite[Definition~12]{CCS}. 

\begin{lemma}\label{inverse-system-canonical-retractions}
Let $K$ be a compact space with weight $\kappa$ and $\mathfrak{s}=(R_s)_{s \in \Sigma}$ be a retractional skeleton on $K$. 
\begin{enumerate}
\item \label{exists-canonical-retraction} There exists a family of sets $(M_\alpha)_{\alpha \in \kappa}$ satisfying conditions (Ra)-(Rd) of \cite[Proposition~17]{CCS} and for each $\alpha \in \kappa$, there exists the canonical retraction $r_\alpha:K \to K$ associated to $M_\alpha$, $K$ and $D(\mathfrak{s})$. Moreover, conditions (R1)-(R10) of \cite[Proposition~17]{CCS} are satisfied.
\item \label{new-part} For every $\alpha \in \kappa$, set $K_\alpha=r_\alpha[K]$ and for every $\alpha, \beta \in \kappa$ with $\alpha \le \beta$, define $r_{\alpha}^{\beta}: K_\beta \to K_\alpha$ as $r_{\alpha}^{\beta}=r_\alpha|_{K_\beta}$. Then $\mathbb{S}=\big((K_\alpha)_{\alpha \in \kappa}, (r_{\alpha}^{\beta}) _{\alpha \le \beta}\big)$ is a continuous inverse system of compact spaces whose bonding maps are retractions, $\big(K, (r_\alpha)_{\alpha \in \kappa}\big)$ is the inverse limit of $\mathbb{S}$ and $w(K_\alpha)\le  \max(\omega, \vert \alpha \vert)$, for every $\alpha \in \kappa$.
\end{enumerate}
\end{lemma}
\begin{proof}
Item \eqref{exists-canonical-retraction} follows from Skolem's Theorem (\cite[Theorem~4]{CCS}) and \cite[Proposition~17]{CCS}. 
Now let us prove \eqref{new-part}. Clearly, (R1) implies that $\mathbb{S}$ is an inverse system whose bonding maps are retractions and (R5) ensures that $w(K_\alpha) \le \max(\omega, \vert \alpha \vert)$. Note that (R2) ensures that the family $\{r_\alpha: \alpha \in \kappa\}$ separates the points of $K$, which implies that $\big(K, (r_\alpha)_{\alpha \in \kappa}\big)$ is the inverse limit of $\mathbb{S}$, since it is a cone for $\mathbb{S}$ and the projections are onto. Finally, using (R3) and a similar argument, we conclude that $\mathbb S$ is continuous. 
\end{proof}

A fundamental concept involved in the proofs of Theorems \ref{semiEberlein-InverseLimit} and \ref{semiEberlein-w1-Shrinking-InverseLimit} is the shrinkingness of a retractional skeleton that was introduced in \cite[Definition~28]{CCS}. Recall that given a retractional skeleton $\mathfrak{s}=(R_s)_{s \in \Sigma}$ on a compact space $K$, a bounded subset $\mathcal{A}$ of $C(K)$ and a subset $D$ of $K$, we say that $\mathfrak{s}$ is \emph{$\mathcal{A}$-shrinking with respect to $D$} if for every $x \in D$ and every increasing sequence $(s_n)_{n \in \omega}$ in $\Sigma$, if $s=\sup_{n \in \omega}s_n$, then $\lim_{n \to \infty} \sup_{f \in \mathcal{A}} \vert f(R_{s_n}(x))-f(R_s(x))\vert=0$. As usual, $C(K)$ denotes the Banach space of real-valued continuous functions defined on the compact space $K$, endowed with the supremum norm.

\begin{theorem}\label{semiEberlein-InverseLimit}
If $K$ is a semi-Eberlein compact space with weight $\kappa$, then there exists a continuous inverse system of compact spaces $\mathbb S=\big((K_\alpha)_{\alpha \in \kappa}, (r_{\alpha}^{\beta})_{\alpha \le \beta}\big)$ such that $K$ is the inverse limit of $\mathbb S$, each $r_{\alpha}^{\beta}$ is a retraction and each $K_\alpha$ is semi-Eberlein with $w(K_\alpha) \le\max(\omega, \vert \alpha \vert)$.
\end{theorem}
\begin{proof}
Let $\mathfrak s=(R_s)_{s \in \Sigma}$ be the retractional skeleton on $K$, $D \subset D(\mathfrak{s})$ be the dense subset of $K$ and $\mathcal{A}$ be the bounded and separating subset of $C(K)$ satisfying conditions (a) and (b) of \cite[Theorem~B(2)]{CCS} and let $(M_\alpha)_{\alpha \in \kappa}$, $(r_\alpha)_{\alpha \in \kappa}$ and $\mathbb{S}=\big((K_\alpha)_{\alpha \in \kappa}, (r_{\alpha}^{\beta})_{\alpha \le \beta}\big)$ be given by Lemma \ref{inverse-system-canonical-retractions}. To conclude the result, it remains to prove that each $K_\alpha$ is semi-Eberlein. Fixed $\alpha \in \kappa$, it follows from (R5) that $\mathfrak{s}_\alpha:=(R_s|_{K_\alpha})_{s \in (\Sigma \cap M_\alpha)_{\sigma}}$ is a retractional skeleton on $K_\alpha$ with $D(\mathfrak{s}_\alpha)=D(\mathfrak s) \cap K_\alpha$. Moreover if we set $D_\alpha=D \cap K_\alpha$, then it is clear that $D_\alpha \subset D(\mathfrak{s}_\alpha)$. Note that (R8) and condition (b) of \cite[Theorem~B(2)]{CCS} ensure that $r_\alpha[D] \subset D_\alpha$, which implies that $D_\alpha$ is dense in $K_\alpha$, since $r_\alpha[D]$ is dense in $K_\alpha$. It follows from condition (b) of \cite[Theorem~B(2)]{CCS}  that for every $x \in D_\alpha$ and every up-directed subset $\Sigma'$ of $(\Sigma \cap M_\alpha)_{\sigma}$ it holds that $\lim_{s \in \Sigma'} R_s(x) \in D_\alpha$. Finally, set $\mathcal{A}_\alpha=\{f|_{K_\alpha}: f \in \mathcal A\}$. It is clear that $\mathcal{A}_\alpha$ is a bounded and separating subset of $C(K_\alpha)$ and it is easy to see that $\mathfrak{s}_\alpha$ is $\mathcal{A}_\alpha$-shrinking with respect to $D_\alpha$. Therefore, \cite[Theorem~B]{CCS} ensures that $K_\alpha$ is semi-Eberlein.
\end{proof}

In Proposition \ref{semiEberleinCharcterizatio-w1} below we show that, for small compacta, the characterization presented in \cite[Theorem~B]{CCS} can be simplified. The result was present in the first version of \cite{CCS} \footnote{Still available at \href{https://arxiv.org/abs/2009.07902v1}{\texttt{arXiv:2009.07902v1}}.}; we are grateful to Marek C\'uth for his permission to insert it here.

\begin{proposition}\label{semiEberleinCharcterizatio-w1}
Let $K$ be a small compact space. Then the following conditions are equivalent:
\begin{enumerate}
    \item[(i)] $K$ is semi-Eberlein.
    \item[(ii)] There exist a bounded and separating subset $\mathcal{A}$ of $C(K)$, a retractional skeleton $\mathfrak{s}=(R_s)_{s \in \Sigma}$ on $K$ and a dense subset $D$ of $K$ such that $D \subset D(\mathfrak{s})$ and $\mathfrak{s}$ is $\mathcal{A}$-shrinking with respect to $D$.
\end{enumerate}
\end{proposition}
\begin{proof}
The fact that (i) implies (ii) follows directly from \cite[Theorem~B]{CCS}. Now assume (ii) and fix $\lambda>1$ such that $\mathcal{A} \subset \lambda B_{C(K)}$. Let $(M_\alpha)_{\alpha \in \omega_1}$ be the family of sets and $(r_\alpha)_{\alpha \in \omega_1}$ be the family of projections given by Lemma \ref{inverse-system-canonical-retractions} \eqref{exists-canonical-retraction} and note that it follows from Skolem's Theorem (\cite[Theorem~4]{CCS}) that we may assume that $\mathcal{A} \in M_0$. 

For each $\alpha < \omega_1$, define $T_\alpha=(\mathcal{A} \cap M_\alpha) \times \{\alpha\}$ and set $T=T_0 \cup \bigcup_{\alpha < \omega_1} T_{\alpha+1}$. Let us show that the mapping $h: K \to [-1,1]^{T}$ defined as \[h(x)(t):=\begin{cases}
\tfrac{1}{2\lambda}(f(x) - f(r_\alpha(x))), & t=(f, \alpha+1),\: f\in \mathcal{A} \cap M_{\alpha+1} ,\\
\tfrac{1}{\lambda}f(r_0(x)), & t=(f,0),\: f\in \mathcal{A} \cap M_0,
\end{cases}\]
is a homeomorphic embedding such that $h\big[D(\mathfrak s)] \subset \Sigma(T)$ and that $(r_\alpha)_{\alpha<\omega_1}$ is a commutative retractional skeleton inducing the set $D(\mathfrak{s})$.
Clearly $h$ is continuous. Let us verify that it is one-to-one. Indeed, if $x, y$ are distinct points from $K$ then by (R7) there is a minimal ordinal $\alpha_0<\omega_1$ for which $r_{\alpha_0}(x)\neq r_{\alpha_0}(y)$ and $\alpha_0=0$ or it is a successor ordinal.
If $\alpha_0=0$, then by (R10) there exists $f \in M_0\cap \mathcal{A}$ such that $f\big(r_0(x)\big) \ne f\big(r_0(y)\big)$ and so we have $h(x)(f,0) \ne h(y)(f,0)$. Otherwise, $\alpha_0=\beta_0+1$ for some $\beta_0<\omega_1$, then by (R10) there exists $f\in \mathcal{A}\cap M_{\beta_0 +1}$ such that $f(x)=f(r_{\beta_0 +1}(x))\neq f(r_{\beta_0 +1}(y))=f(y)$. Therefore, since $r_{\beta_0}(x)=r_{\beta_0}(y)$, we obtain that $h(x)(f,\beta_0 +1) \neq h(y)(f,\beta_0 +1)$. Note that it follows from (R1), (R2), (R3) and (R5) that $(r_\alpha)_{\alpha<\omega_1}$ is a retractional skeleton on $K$, since (Rb) ensures that $M_\alpha$ is countable, for every $\alpha < \omega_1$ (where on $\omega_1$ we consider the ordering $\alpha\leq_* \beta$ if and only if $M_\alpha\subset M_\beta$). Moreover by \cite[Theorem~15(ii)(c)]{CCS}  and \cite[Lemma 3.2]{C} we have that $\bigcup_{\alpha<\omega_1} r_\alpha[K] = D(\mathfrak s)$. Now let us show that $h[D(\mathfrak{s})] \subset \Sigma(T)$.  
If $x\in D(\mathfrak{s})$, then there exists a minimal $\alpha_0<\omega_1$ such that $r_{\alpha_0}(x)=x$. We claim that
\[
\suppt \big(h(x)\big)=\{t\in T \colon h(x)(t)\neq 0\}\subset T_0 \cup \bigcup_{\alpha<\alpha_0+1} T_{\alpha+1}.
\]
Indeed, if $t \in T$ with $t \notin T_0 \cup \bigcup_{\alpha<\alpha_0+1} T_{\alpha+1}$, then there exist $\alpha>\alpha_0 $ and $f\in M_{\alpha +1}\cap \mathcal{A}$ such that $t=(f,\alpha +1)$. Thus,
\[h(x)(t)=\tfrac{1}{2\lambda}\big(f(x)-f\big(r_\alpha(x)\big)\big)=0,\]
since $\alpha>\alpha_0$ implies that $r_\alpha(x)=x$. This proves the claim.
Finally the fact that $h[D(\mathfrak s)]\subset \Sigma(T)$, follows by observing that $T_0$ is countable as well as $T_{\alpha +1}$, for every $\alpha< \omega_1$.

Now consider the mapping $\Phi:K\to [-1,1]^{T}$ given by
\[
\Phi(x)(f_n^{\alpha+1},\alpha+1):=\frac{1}{n}h(x)(f_n^{\alpha+1}, \alpha+1) \quad \text{and} \quad \Phi(x)(f_n^{0},0):=\frac{1}{n}h(x)(f_n^{0},0),
\]
where $\mathcal{A} \cap M_0=(f_{n}^{0})_{n<|\mathcal{A}\cap M_0|}$ and $\mathcal{A}\cap M_{\alpha+1}=(f_{n}^{\alpha+1})_{n<|\mathcal{A}\cap M_{\alpha+1}|}$, for every $\alpha < \omega_1$. It is clear that $\Phi$ is a homeomorphic embedding and that $\Phi[D(\mathfrak{s})]\subset \Sigma(T)$. In order to conclude that $K$ is semi-Eberlein, it remains to show that $\Phi[D]\subset c_0(T)$. Since $\mathfrak s$ is $\mathcal{A}$-shrinking with respect to $D$, using that the mapping $\omega_1\ni \alpha\mapsto \sup (M_\alpha\cap\Sigma)\in\Sigma$ is increasing and \cite[Theorem~15(ii)(c)]{CCS}, we obtain that $(r_\alpha)_{\alpha <\omega_1}$ is $\mathcal{A}$-shrinking with respect to $D$. Fix $x\in D$ and note that to prove that $\Phi(x) \in c_0(T)$, it is enough to show that, for every $\varepsilon>0$, the following set is finite
\[\Lambda:= \{\alpha <\omega_1:|h(x)(t)|>\varepsilon \text{ for some }t\in T_{\alpha+1}\}.\]
Suppose by contradiction that $\Lambda$ is infinite. Take a strictly increasing sequence $(\alpha_k)_{k\in\omega}$ from $\Lambda$ and functions $f_k\in \mathcal{A}\cap M_{\alpha_k+1}$ such that  $|h(x)(f_k,\alpha_k +1)|>\varepsilon$. Let $\alpha=\sup_{k\in\omega}\alpha_k$ and fix $k \in \omega$. Then
\begin{equation*}
    \begin{split}
        \rho_{\mathcal{A}}(r_{\alpha_k}(x),r_{\alpha}(x))&=\sup_{f\in\mathcal{A}}|\langle r_{\alpha_k}(x) - r_{\alpha}(x),f\rangle|  \geq |\langle r_{\alpha_k}(x) - r_{\alpha}(x),f_k\rangle|\\ &=|f_k(r_{\alpha}(x))-f_k(r_{\alpha_k}(x))|.
    \end{split}
\end{equation*}
Observing that $M_{\alpha_k+1}\subset M_{\alpha}$, (R10) ensures that $f_k(r_{\alpha}(x))=f_k(x)$. 
Therefore we obtain
\[\rho_{\mathcal{A}}(r_{\alpha_k}(x),r_{\alpha}(x))\geq |f_k(x)-f_k(r_{\alpha_k}(x))|>2\lambda \varepsilon.\]
But this contradicts the fact that the skeleton $(r_\alpha)_{\alpha<\omega_1}$ is $\mathcal{A}$-shrinking with respect to $D$.
\end{proof}

Proposition \ref{semiEberleinCharcterizatio-w1} allows us to obtain in Theorem \ref{semiEberlein-w1-Shrinking-InverseLimit} an interesting characterization of small semi-Eberlein spaces in terms of inverse limits. In order to do so, we introduce Definition \ref{Shrininking-InverseLimit} that is the translation of shrinkingness to the context of inverse limits. First let us recall the notion of right inverse of an inverse system. Given an inverse system $\mathbb S=\big((K_s)_{s \in \Sigma}, (p_{s}^{t}) _{s\le t}\big)$, we say that a family of continuous maps $\mathbb I=\{i_{s}^{t}: K_s \to K_t, s \le t\}$ is a \emph{right inverse} of $\mathbb{S}$ if $i_{s}^{t}$ is a right inverse of $p_{s}^{t}$, for every $s \le t$ and $i_{s}^{r}=i_{t}^{r} \circ i_{s}^{t}$, for every $s \le t \le r$. In this case \cite[Lemma~3.1]{KM} ensures that if $\big(K, (p_s)_{s \in \Sigma}\big)$ is the inverse limit of $\mathbb{S}$, then there exists a unique family of continuous maps $\{i_s: K_s \to K, s \in \Sigma\}$ such that $i_s$ is a right inverse of $p_s$, for every $s \in \Sigma$ and $i_s=i_t \circ i_{s}^{t}$, for every $s \le t$. We say that $\{i_s: s \in \Sigma\}$ is the \emph{right inverse of $\big(K, (p_{s})_{s \in \Sigma}\big)$ with respect to $\mathbb I$}.

\begin{definition}\label{Shrininking-InverseLimit}
Let $\mathbb S=\big((K_s)_{s \in \Sigma}, (p_{s}^{t})_{s\le t}\big)$ be a $\sigma$-complete inverse system, $\mathbb I$ be a right inverse of $\mathbb S$, $\big(K, (p_{s})_{s \in \Sigma}\big)$ be the inverse limit of $\mathbb S$, $\mathcal A$ be a bounded subset of $C(K)$ and $D$ be a subset of $K$. We say that $\big(K, (p_{s})_{s \in \Sigma}\big)$ is \emph{($\mathcal{A}, \mathbb I)$-shrinking with respect to $D$} if for every increasing sequence $(s_n)_{n \in \omega}$ in $\Sigma$ with $s=\sup_{n \in \omega}s_n$ and every $x \in D$ it holds that 
$\lim_{n \to \infty} \sup_{f \in \mathcal{A}} \vert f \circ i_{s_n} \circ p_{s_n}(x)-f \circ i_s \circ p_s(x)\vert=0$,
where $\{i_s: s \in \Sigma\}$ is the right inverse of $\big(K, (p_{s})_{s \in \Sigma}\big)$ with respect to $\mathbb I$.
\end{definition}

\begin{theorem}\label{semiEberlein-w1-Shrinking-InverseLimit}
Let $K$ be a small compact space. Then the following conditions are equivalent:
\begin{enumerate}
    \item  $K$ is semi-Eberlein.
    \item There exist a bounded and separating subset $\mathcal A$ of $C(K)$, a dense subset $D$ of $K$, a continuous inverse system of compact metric spaces $\mathbb{S}=\big((K_\alpha)_{\alpha \in \omega_1}, (r_{\alpha}^{\beta})_{\alpha \le \beta}\big)$ with a right-inverse $\mathbb I=\{i_{\alpha}^{\beta}: \alpha \le \beta\}$ and a family of retractions $\{r_\alpha: K \to K_\alpha, \alpha \in \omega_1\}$ such that $\big(K, (r_\alpha)_{\alpha \in \omega_1}\big)$ is the inverse limit of $\mathbb{S}$, it is $(\mathcal A, \mathbb I)$-shrinking with respect to $D$ and $D \subset \bigcup_{\alpha \in \omega_1} i_\alpha[K_\alpha]$, where $\{i_\alpha: \alpha \in \omega_1\}$ is the right inverse of $\big(K, (r_\alpha)_{\alpha \in \omega_1}\big)$ with respect to $\mathbb I$.
\end{enumerate}
\end{theorem}
\begin{proof}
Assume that $K$ is semi-Eberlein. Let $\mathfrak s=(R_s)_{s \in \Sigma}$ be the retractional skeleton on $K$, $D \subset D(\mathfrak{s})$ be the dense subset of $K$ and $\mathcal{A}$ be the bounded and separating subset of $C(K)$ given by Proposition \ref{semiEberleinCharcterizatio-w1} and let $(M_\alpha)_{\alpha \in \omega_1}$, $(r_\alpha)_{\alpha \in \omega_1}$ and $\mathbb{S}=\big((K_\alpha)_{\alpha \in \omega_1}, (r_{\alpha}^{\beta}) _{\alpha \le \beta}\big)$ be given by Lemma \ref{inverse-system-canonical-retractions}. It is clear that each $K_\alpha$ is metrizable and that $\mathbb I=\{i_{\alpha}^{\beta}: \alpha \le \beta\}$ is a right inverse of $\mathbb S$, where $i_{\alpha}^{\beta}$ is the inclusion of $K_\alpha$ into $K_\beta$, for every $\alpha \le \beta$. Moreover, for each $\alpha \in \omega_1$, if we define $i_\alpha$ as the inclusion of $K_\alpha$ into $K$, then $\{i_\alpha: \alpha \in \omega_1\}$ is the right inverse of $\big(K, (r_\alpha)_{\alpha \in \omega_1}\big)$ with respect to $\mathbb I$. Note that it follows from (R1), (R2), (R3) and (R5) that $(r_\alpha)_{\alpha\in\omega_1}$ is a retractional skeleton on $K$. Therefore, using \cite[Theorem~15(ii)(c)]{CCS} and \cite[Lemma~3.2]{C}, we conclude that $\bigcup_{\alpha \in \omega_1}r_\alpha[K]=D(\mathfrak{s})$ and thus $D \subset \bigcup_{\alpha \in \omega_1} r_\alpha[K]$. Finally, the fact that $\big(K, (r_\alpha)_{\alpha \in \omega_1}\big)$ is $(\mathcal A, \mathbb I)$-shrinking with respect to $D$ follows from the $\mathcal{A}$-shrinkingness of $\mathfrak{s}$ with respect to $D$ and \cite[Theorem~15(ii)(c)]{CCS}, having in mind that $(M_\alpha)_{\alpha \in \omega_1}$ is an increasing and continuous family of countable sets. Now assume that (2) holds and for each $\alpha \in \omega_1$, define $q_\alpha: K \to K$ as $q_\alpha=i_\alpha \circ r_\alpha$. 
We claim that $\mathfrak{s}=(q_\alpha)_{\alpha \in \omega_1}$ is a retractional skeleton on $K$. Indeed, conditions (i) and (ii) of Definition \ref{def: r-skeleton} are clearly satisfied. Note that \cite[Lemma~3.3~(2)]{KM} ensures that condition (iv) holds. Finally, condition (iii) follows from the continuity of $\mathbb{S}$ and \cite[Lemma~3.3~(2)]{KM}. Since $D \subset \bigcup_{\alpha \in \omega_1} i_\alpha[K_\alpha]$, we have that $D \subset D(\mathfrak{s})$ and it is easy to see that $(q_\alpha)_{\alpha \in \omega_1}$ is $\mathcal A$-shrinking with respect to $D$. Therefore Propostion \ref{semiEberleinCharcterizatio-w1} ensures that $K$ is semi-Eberlein.
\end{proof}

\section{A special class of semi-Eberlein compacta}\label{sec:RS}

This section is dedicated to the introduction and study of the class $\mathcal{RS}$. Following \cite{KL}, we say that a function $f:X \to Y$ between topological spaces is \emph{semi-open} if $f[U]$ has nonempty interior, for every nonempty open subset $U$ of $X$.

\begin{definition}
We say that a compact space $K$ belongs to $\mathcal{RS}$ if $K$ is the inverse limit of a continuous inverse system of compact metric spaces $\mathbb{S}=\big((K_\alpha)_{\alpha \in \omega_1}, (p_{\alpha}^{\beta}) _{\alpha \le \beta}\big)$ such that $p_{\alpha}^{\beta}$ is a semi-open retraction, for every $\alpha \le \beta$.
\end{definition}

In Corollary \ref{weaker-definition} we show that in the definition of $\mathcal{RS}$ it is enough to require that $p_{\alpha}^{\alpha+1}$ is a semi-open retraction, for every $\alpha \in \omega_1$.

\begin{lemma}\label{quotient-semiopen}
Let $X$, $Y$ and $Z$ be topological spaces, $q:X \to Y$ be a continuous and onto map and $h:X \to Z$ be a semi-open map. If there exists $\bar{h}:Y \to Z$ such that $\bar{h} \circ q=h$, then $\bar{h}$ is semi-open.
\end{lemma}
\begin{proof}
Let $U$ be a nonempty open subset of $Y$. Since $q$ is continuous and onto, we have that $q^{-1}[U]$ is a nonempty and open subset of $X$. Thus 
$h\big[q^{-1}[U]\big]=\bar{h}[U]$ has nonempty interior, since $h$ is semi-open.
\end{proof}

\begin{lemma}\label{ps-semiopen-equiv-pst-semiopen}
Let $\mathbb{S}=\big((K_s)_{s \in \Sigma}, (p_{s}^{t})_{s \le t}\big)$ be an inverse system of compact spaces and $\big(K, (p_s)_{s \in \Sigma}\big)$ be the inverse limit of $\mathbb{S}$. Then the following conditions are equivalent:
\begin{enumerate}
    \item[(a)] $p_s$ is semi-open, for every $s \in \Sigma$.
    \item[(b)] $p_{s}^{t}$ is semi-open, for every $s \le t$.
\end{enumerate}
\end{lemma}
\begin{proof}
To see that (a) implies (b), fix $s \le t$ and apply Lemma \ref{quotient-semiopen} to $q=p_t$, $h=p_s$ and $\bar{h}=p_{s}^{t}$. Now assume (b) and fix $s_0 \in \Sigma$. Since $\{s \in \Sigma: s_0 \le s\}$ is a cofinal subset of $\Sigma$, \cite[Proposition~2.5.5]{E} ensures that $\bigcup_{s\ge s_0} \{p_s^{-1}[W]: W \ \text{is open in} \ K_s\}$ is an open basis of $K$ and therefore to conclude that $p_{s_0}$ is semi-open, it is enough to show that for every $s \ge s_0$ and every nonempty open subset $W$ of $K_s$ the set $p_{s_0}\big[p_s^{-1}[W]\big]$ has nonempty interior. This follows from the fact that $p_{s_0}^{s}$ is semi-open and $p_{s_0}^{s}[W] \subset p_{s_0}\big[p_s^{-1}[W]\big]$.
\end{proof}

\begin{proposition}\label{consective-semiopen-all-semiopen}
Let $\kappa$ be a cardinal and $\mathbb{S}=\big((K_\alpha)_{\alpha \in \kappa}, (p_{\alpha}^{\beta})_{\alpha \le \beta} \big)$ be a continuous inverse system of compact spaces. If $p_{\alpha}^{\alpha+1}$ is semi-open, for every $\alpha \in \kappa$, then $p_{\alpha}^{\beta}$ is semi-open, for every $\alpha \le \beta$.
\end{proposition}
\begin{proof}
Let us prove by induction on $\alpha \in \kappa$ that $p_{\gamma_1}^{ \gamma_2}$ is semi-open, for every $\gamma_1<\gamma_2 \le \alpha$.
Suppose that the result holds for $\alpha$ and let $\gamma_1<\gamma_2 \le \alpha+1$. If $\gamma_2 \le \alpha$, then the result follows from the induction hypothesis. Otherwise, we have that $\gamma_2=\alpha+1$ and thus $p_{\gamma_1}^{\gamma_2}$ is semi-open, since $p_{\gamma_1}^{\alpha+1} =p_{\gamma_1}^{\alpha}\circ p_{\alpha}^{\alpha+1}$ and $p_{\gamma_1}^{ \alpha}$ and $p_{\alpha}^{\alpha+1}$ are semi-open. 
Now fix a limit ordinal $\alpha \in \kappa$ and assume that the result holds for every ordinal strictly smaller than $\alpha$. Note that the induction hypothesis ensures that $p_{\gamma_1}^{\gamma_2}$ is semi-open, for every $\gamma_1<\gamma_2< \alpha$. Moreover, it follows from the continuity of $\mathbb{S}$ that $\big(K_\alpha, (p_{\gamma}^{\alpha}) _{\gamma \in \alpha}\big)$ is the inverse limit of $\mathbb{S}|_{\alpha}$ and therefore Lemma \ref{ps-semiopen-equiv-pst-semiopen} ensures that $p_{\gamma}^{\alpha}$ is semi-open, for every $\gamma \in \alpha$.
\end{proof}

\begin{corollary}\label{weaker-definition}
If $\mathbb{S}=\big((K_\alpha)_{\alpha \in \omega_1}, (p_{\alpha}^{\beta}) _{\alpha \le \beta}\big)$ is a continuous inverse system of compact metric spaces such that $p_{\alpha}^{\alpha+1}$ is a semi-open retraction, for every $\alpha \in \omega_1$, then the inverse limit of $\mathbb{S}$ belongs to $\mathcal{RS}$.
\end{corollary}
\begin{proof}
The result follows from Proposition \ref{consective-semiopen-all-semiopen}, having in mind that \cite[Lemma~3.2]{KM} ensures that $p_{\alpha}^{\beta}$ is a retraction, for every $\alpha \le \beta$.
\end{proof}

Given the tight connection between inverse limits and retractional skeletons already explored in Section \ref{sec: 2}, we can now provide a useful equivalent description of the class $\mathcal{RS}$ via retractional skeletons.

\begin{lemma}\label{RS-skeleton}
If $K$ is a compact space, then the following conditions are equivalent:
\begin{enumerate}
    \item $K$ belongs to $\mathcal{RS}$
    \item $K$ admits a retractional skeleton $(R_\alpha)_{\alpha \in \omega_1}$ such that $R_\alpha:K \to R_\alpha[K]$ is semi-open, for every $\alpha \in \omega_1$.
\end{enumerate}
\end{lemma}
\begin{proof}
Assume that $K$ belongs to $\mathcal{RS}$ and let $\mathbb{S}=\big( (K_\alpha)_{\alpha \in \omega_1}, (p_{\alpha}^{\beta})_{\alpha \le \beta}\big)$ be a continuous inverse system of compact metric spaces such that each $p_{\alpha}^{\beta}$ is a semi-open retraction and let $\{p_\alpha:K\to K_\alpha, \alpha \in \omega_1\}$ be such that $\big(K, (p_\alpha)_{\alpha \in \omega_1}\big)$ is the inverse limit of $\mathbb{S}$. Let $\{i_\alpha:K_\alpha \to K, \alpha \in \omega_1\}$ be the family of maps whose existence is ensured by \cite[Lemma~3.1 and Lemma~3.2]{KM} and for every $\alpha \in \omega_1$ set $R_\alpha:=i_\alpha \circ p_\alpha:K \to K$. It follows from the continuity of $\mathbb{S}$ and \cite[Lemma~3.3]{KM} that $(R_\alpha)_{\alpha \in \omega_1}$ is a retractional skeleton on $K$. Moreover, since Lemma \ref{ps-semiopen-equiv-pst-semiopen} ensures that each $p_{\alpha}$ is semi-open, it is easy to see that $R_\alpha:K \to R_\alpha[K]$ is semi-open.
Now assume (2) and for each $\alpha \le \beta$, define the retraction $p_{\alpha}^{\beta}:R_\beta[K] \to R_\alpha[K]$ as $p_{\alpha}^{\beta} :=R_\alpha|_{R_\beta[K]}$. It follows from conditions (i) and (iii) of Definition \ref{def: r-skeleton} and \cite[Lemma~3.4]{KM} that $\big(K, (R_\alpha)_{\alpha \in \omega_1}\big)$ is the inverse limit of the continuous inverse system of compact metric spaces $\big((R_\alpha[K]) _{\alpha \in \omega_1}, (p_{\alpha}^{\beta})_{\alpha \le \beta}\big)$. Finally, the fact that each $p_{\alpha}^{\beta}$ is semi-open follows from Lemma \ref{ps-semiopen-equiv-pst-semiopen}.
\end{proof}

Quite clearly, compact metric spaces and the cubes $2^{\omega_1}$ and $[-1,1]^{\omega_1}$ are examples of compacta that belong to $\mathcal{RS}$. More generally, retracts of such cubes also belong to $\mathcal{RS}$, \cite[Corollary 4.3]{KL}. On the other hand, we shall now show that $\mathcal{RS}$ is indeed a proper subclass of the class of small semi-Eberlein compacta. Given a set $\Gamma$ and $x \in 2^{\Gamma}$, we denote by $\suppt(x)$ the support of $x$, i.e., $\suppt(x)=\{\gamma \in \Gamma \colon x(\gamma) \ne 0\}$. Moreover, given a topological space $K$, we denote the set of its isolated points by $I(K)$.

\begin{theorem}\label{key-for-scattered}
Let $K \subset 2^{\Gamma}$ be a compact space. Assume that there exists an uncountable subset $A$ of $I(K)$ such that $\suppt(x)$ is finite, for every $x \in A$. Then $K$ does not belong to $\mathcal{RS}$.
\end{theorem}

\begin{proof}
Suppose by contradiction that $K \in \mathcal{RS}$. Let $(R_\alpha)_{\alpha \in \omega_1}$ be the retractional skeleton given by Lemma \ref{RS-skeleton} and for each $\alpha \in \omega_1$, set $K_\alpha=R_\alpha[K]$. In what follows, we identify subsets of $\Gamma$ with their characteristic functions. It follows from the $\Delta$-system Lemma that there exist a finite subset $\Delta$ of $\Gamma$ and an uncountable subset $B$ of $A$ such that $b_1 \cap b_2=\Delta$, for every $b_1, b_2 \in B$ with $b_1 \ne b_2$. We claim that there exists $\alpha \in \omega_1$ such that $\Delta \in K_\alpha \setminus I(K_\alpha)$. 
Indeed, it is easy to see that any injective sequence of elements of $B$ converges to $\Delta$ in $2^{\Gamma}$. Fix an injective sequence $(b_n)_{n \in \omega}$ of elements of $B$. Since $\bigcup_{\alpha \in \omega_1} K_\alpha$ is dense in $K$, we have that $I(K) \subset \bigcup_{\alpha \in \omega_1} K_\alpha$ and therefore, for every $n \in \omega$, there exists $\alpha_n \in \omega_1$ such that $b_n \in K_{\alpha_n}$. If $\alpha=\sup_{n \in \omega} \alpha_n$, then $b_n \in K_\alpha$, for every $n \in \omega$, which implies that $\Delta \in K_\alpha$ and of course $\Delta \notin I(K_\alpha)$. It follows from the fact that $R_\alpha:K \to K_\alpha$ is semi-open that $R_\alpha(z) \in I(K_\alpha)$, for every $z \in I(K)$. Thus we have that $B=\bigcup_{x \in I(K_\alpha)} R_{\alpha}^{-1}[\{x\}] \cap B$, which implies that there exists $x \in I(K_\alpha)$ such that $R_{\alpha}^{-1}[\{x\}] \cap B$ is uncountable, since $I(K_\alpha)$ is countable. Let $(b_n)_{n \in \omega}$ be an injective sequence of elements of $R_{\alpha}^{-1}[\{x\}] \cap B$. As argued above, we have that $(b_n)_{n \in \omega}$ converges to $\Delta$ in $K$ and thus $\big(R_\alpha(b_n)\big)_{n \in \omega}$ converges to $R_\alpha(\Delta)$. But this is a contradiction, because $R_\alpha(\Delta)=\Delta$ and $R_\alpha(b_n)=x$, for every $n \in \omega$.
\end{proof}

We recall that for a set $\Gamma$, we define $\Sigma(\Gamma)=\{x \in \mathbb R^{\Gamma}: \suppt(x) \ \text{is countable}\}$ and that the small $\sigma$-product of real lines is defined by $\sigma(\Gamma)=\{x\in \R^\Gamma: \suppt(x) \mbox{ is finite}\}$. Plainly, every compact space $K \subset \mathbb{R}^{\Gamma}$ such that $\sigma(\Gamma)\cap K$ is dense in $K$ is a fortiori semi-Eberlein.

\begin{corollary}\label{scattered-corson}
If $K\subset 2^{\Gamma}$ is a nonmetrizable scattered compact space such that $\sigma(\Gamma)\cap K$ is dense in $K$, then $K$ does not belong to $\mathcal{RS}$. In particular, nonmetrizable scattered Corson compacta do not belong to $\mathcal{RS}$.
\end{corollary}
\begin{proof}
According to Theorem \ref{key-for-scattered} and using that $I(K) \subset \sigma(\Gamma)\cap K$, in order to conclude that $K$ does not belong to $\mathcal{RS}$, it is enough to show that $I(K)$ is uncountable. Note that if  $I(K)$ is countable, then $K \subset \Sigma(\Gamma)$, since $K$ is scattered, $I(K) \subset \Sigma(\Gamma)$ and $\Sigma(\Gamma)$ is countably closed in $\mathbb{R}^{\Gamma}$. But this implies that $K$ is a separable Corson compact space and therefore metrizable. Now assume that $K$ is a nonmetrizable scattered Corson compact space. Then the result follows from \cite[Corollary~1]{A}, since it ensures that $K$ is strongly Eberlein, i.e., we may assume that $K \subset 2^{\Gamma} \cap \sigma(\Gamma)$, for some set $\Gamma$.
\end{proof}

Quite on the opposite extreme of the connectedness spectrum, we now show that the unit ball of the Hilbert space $\ell_2(\omega_1)$, endowed with the weak topology, does not belong to $\mathcal{RS}$. Recall that $\ell_2(\omega_1)$ denotes the Hilbert space $\{x\colon \omega_1 \to\R  \colon \sum_{\gamma \in \omega_1} \vert x(\gamma)\vert^2<\infty\}$, endowed with the norm $\n_2$ given by $\|x\|_2:=(\sum_{\gamma \in \omega_1} \vert x(\gamma) \vert^2)^{1/2}$. It follows from the reflexivity of $\ell_2(\omega_1)$ that its closed unit ball $B_{\ell_2(\omega_1)}$, endowed with the weak topology, is compact and thus it is an Eberlein compact space. Finally, recall that the weak topology and the product topology coincide on $B_{\ell_2(\omega_1)} \subset [-1, 1]^{\omega_1}$.

\begin{theorem}\label{BallHilbert-Not-RS}
If $K=B_{\ell_2(\omega_1)} \subset [-1,1]^{\omega_1}$, then $K$ does not belong to $\mathcal{RS}$.
\end{theorem}

\begin{proof} Towards a contradiction, assume that $K\in\mathcal{RS}$. Then by Lemma \ref{RS-skeleton} there exists a retractional skeleton  $\mathfrak{s}=(R_{\alpha})_{\alpha\in\omega_1}$ on $K$ such $R_{\alpha}:K\to R_{\alpha}[K]$ is semi-open, for every $\alpha\in\omega_1$. Moreover it is easy to see that $\mathfrak{s}'=(Q_\alpha)_{\alpha \in \omega_1}$ is also a retractional skeleton on $K$, where $Q_\alpha:K \to K$ is given by  
$$Q_\alpha(x)(\gamma)=
\begin{cases}
x(\gamma) & \gamma<\alpha \\
0 & \gamma \ge \alpha
\end{cases}
\qquad (x \in K).$$
Since $K$ is Eberlein, \cite[Theorem 3.11]{C} ensures that $D(\mathfrak{s})=D(\mathfrak{s}')=K$. Therefore it follows from \cite[Theorem 21]{CCS}
applied to $\mathfrak{s}$ and $\mathfrak{s}'$ that there exist  $\alpha, \beta< \omega_1$ such that $\alpha$ is infinite and $Q_\alpha =R_\beta$, which implies that $Q_\alpha: K \to Q_\alpha[K]$ is semi-open. However, we shall show that this is not the case. Fix $\gamma\in \omega_1$ with $\alpha<\gamma$ and consider the nonempty open subset $W=\{x\in K\colon x(\gamma)\neq 0\}$ of $K$. Let $U$ be a nonempty basic open subset of $Q_\alpha[K]$; then there exist $u\in Q_\alpha[K]$, $\alpha_1,\dots,\alpha_n\in \alpha$ and $\e>0$ such that
$U=\left\{ x \in Q_\alpha[K] \colon |x(\alpha_i)- u(\alpha_i)|< \e,\ i=1,\dots,n \right\}$. Note that there exists $x\in U$ such that $\|x\|_2=1$. On the other hand, for every $y \in W$, we have that
$$1\ge \|y\|_2^2\ge \|Q_\alpha(y)\|_2^2 + |y(\gamma)|^2 > \|Q_\alpha(y)\|_2^2,$$
which implies that $\|Q_\alpha(y)\|_2 < 1$. Therefore, there is no $y\in W$ with $Q_\alpha(y)=x$ and thus $U\not\subset Q_\alpha[W]$. Since $U$ was arbitrary, we conclude that $Q_\alpha[W]$ has empty interior in $Q_\alpha[K]$, as desired.
\end{proof}

\begin{remark} For the reader inclined to Banach space theory, we observe that the above result actually holds in a more general setting. Indeed, a similar argument shows the following: if $X$ is a WLD Banach space and $X^*$ is strictly convex, then the dual unit ball of $X$, endowed with the weak$^*$ topology, does not belong to $\mathcal{RS}$. In particular, this applies to the unit ball of $\ell_p(\omega_1)$, endowed with the weak topology, for $1<p<\infty$. Finally, the same proof as for $\ell_2(\omega_1)$ also gives that the unit ball of $\ell_1(\omega_1)$, endowed with the weak$^*$ topology induced by $c_0(\omega_1)$, does not belong to $\mathcal{RS}$ either.
\end{remark}

\subsection{Stability results}\label{subsec: stability-RS} The class $\mathcal{RS}$ enjoys several stability properties. Needless to say, such properties yield several further examples of compacta in this class. 

\begin{proposition}\label{hereditary-clopen}
If $K$ belongs to $\mathcal{RS}$ and $L$ is a clopen subset of $K$, then $L$ belongs to $\mathcal{RS}$.
\end{proposition}

\begin{proof}
Let $\mathfrak{s}=(R_\alpha)_{\alpha \in \omega_1}$ be the retractional skeleton given by Lemma \ref{RS-skeleton}. Since $L$ is open, we have that $D(\mathfrak{s}) \cap L$ is dense in $L$, therefore it follows from \cite[Lemma~3.5]{C} that the set $T=\{\alpha \in \omega_1: R_\alpha[L] \subset L\}$ is a cofinal and $\sigma$-closed subset of $\omega_1$ and $(R_\alpha|_L)_{\alpha \in T}$ is a retractional skeleton on $L$. Clearly the fact that $L$ is open ensures that $R_\alpha|_L:L \to R_\alpha[L]$ is semi-open, for every $\alpha \in T$. Therefore, the result follows from Lemma \ref{RS-skeleton}, since $T$ is order-isomorphic to $\omega_1$. 
\end{proof}

\begin{remark}
Note that the class $\mathcal{RS}$ is not stable for closed subspaces in general. For instance, every small compact space embeds homeomorphically into the cube $[0,1]^{\omega_1}$ and $[0,1]^{\omega_1}$ belongs to $\mathcal{RS}$.
\end{remark}

\begin{proposition}\label{closed-product}
If $\{K^{\alpha}: \alpha\in\omega_1\}$ is a family of elements of  $\mathcal{RS}$, then $K=\prod_{\alpha\in\omega_1}K^{\alpha}$ belongs to $\mathcal{RS}$.
\end{proposition}
\begin{proof}
If any $K^\alpha$ is empty, then the result is trivial. For every $\alpha \in \omega_1$, let $(R^\alpha_\beta)_{\beta \in \omega_1}$ be the retractional skeleton on $K^\alpha$ given by Lemma \ref{RS-skeleton} and for each $\alpha \in \omega_1$, fix an element $x_\alpha \in R_0^\alpha[K^\alpha]$. Fixed $\beta \in \omega_1$, define $P_\beta: K \to K$ as $P_\beta(y)(\alpha)=R_\beta^\alpha(y(\alpha))$ if $\alpha<\beta$ and $P_\beta(y)(\alpha)=x_{\alpha}$ otherwise, for every $y=(y(\alpha))_{\alpha \in \omega_1} \in K$. It is straightforward to check that $(P_\beta)_{\beta \in \omega_1}$ is a retractional skeleton on $K$. Now fix $\beta \in \omega_1$ and let us show that $P_\beta:K \to P_\beta[K]$ is semi-open. Let $F$ be a finite subset of $\omega_1$ and consider the basic open set $U=\prod_{\alpha \in F} U^\alpha \times \prod_{\alpha \notin F} K^\alpha$, where $U^\alpha$ is a nonempty open subset of $K^\alpha$, for every $\alpha \in F$. Note that $P_\beta[U]=\prod_{\alpha \in F \cap \beta} R^\alpha_\beta[U^\alpha] \times \prod_{\alpha \in \beta \setminus F} R_\beta^\alpha[K^\alpha] \times \prod_{\alpha \ge \beta} \{x_\alpha\}$ and that $P_\beta[K]=\prod_{\alpha \in \beta} R_\beta^\alpha[K^\alpha] \times \prod_{\alpha \ge \beta} \{x_\alpha\}$. 
Therefore, it is easy to see that $P_\beta[U]$ has nonempty interior in $P_\beta[K]$, since $R_\beta^\alpha: K^\alpha \to R_\beta^\alpha[K^\alpha]$ is semi-open, for every $\alpha \in F \cap \beta$. The result follows from Lemma \ref{RS-skeleton}.
\end{proof}

Given a family of topological spaces $\{X^i: i \in I\}$, we denote by $\bigsqcup_{i \in I} X^i$ its topological sum.

\begin{proposition}
If $\{K^n: n\in\omega\}$ is a family of elements of $\mathcal{RS}$, then the one-point compactification $K=\bigsqcup_{n\in\omega}K^n \cup \{\infty\}$ of $\bigsqcup_{n\in\omega}K^n$ belongs to $\mathcal{RS}$. In particular, every finite topological sum of elements of $\mathcal{RS}$ belongs to $\mathcal{RS}$.
\end{proposition}

\begin{proof}
For every $n \in \omega$, let $(R^{n}_\beta)_{\beta \in \omega_1}$ be the retractional skeleton on $K^n$ given by Lemma \ref{RS-skeleton}. For each $\beta\in\omega_1$, define $P_{\beta}:K\to K$ by $P_\beta (x)= R_{\beta}^{n}(x)$, if $x\in K^n$ and $P_{\beta}(\infty)=\infty$. We claim that $(P_\beta)_{\beta \in \omega_1}$ is a retractional skeleton on $K$. In fact it is straightforward to check that it satisfies conditions (ii), (iii) and (iv) of Definition \ref{def: r-skeleton} and condition (i) follows from the fact that the one-point compactification of a locally compact and $\sigma$-compact metric space is metrizable \cite[Theorem~5.3]{Ke}. Now fix $\beta \in \omega_1$ and let us show that $P_{\beta}:K\to P_{\beta}[K]$ is semi-open. If $U$ is a nonempty open subset of $K$, then there exists $n\in\omega$ such that $U\cap K^n$ is a nonempty open subset of $K^n$. Therefore, the semi-openness of $R_{\beta}^{n}:K^n \to R_{\beta}^{n}[K^n]$ easily implies that $P_\beta[U \cap K^n]$ has nonempty interior in $P_\beta[K]$. Using Lemma \ref{RS-skeleton}, we conclude the proof. Now fix a finite family $\{K^i: i=1, \ldots, k\}$ of elements of $\mathcal{RS}$ and note that $\bigsqcup_{i=1}^{k} K^i$ is a clopen subset of the one-point compactification of the topological sum of a countable family of elements of $\mathcal{RS}$ and therefore, Proposition \ref{hereditary-clopen} ensures that $\bigsqcup_{i=1}^{k} K^i$ belongs to $\mathcal{RS}$.
\end{proof}

Next, we study the stability of $\mathcal{RS}$ for continuous images. In Theorem \ref{continuous-Valdivia-image-belongs-class} we show that if a Valdivia compact space $L$ is a continuous and semi-open image of an element of $\mathcal{RS}$, then $L$ belongs to $\mathcal{RS}$. As a consequence of this result, we conclude in Corollary \ref{closed-semiopen-retraction} that the class $\mathcal{RS}$ is stable under semi-open retractions. Proposition \ref{ft-semiopen-big} is the key to establish Theorem \ref{continuous-Valdivia-image-belongs-class} and its proof relies mostly on the theory of suitable models presented in \cite[Section~3.1]{CCS}. Recall that a subset $T$ of a $\sigma$-complete and up-directed partially ordered set $\Sigma$ is said to be \emph{$\sigma$-closed} in $\Sigma$ if the supremum of $M$ belongs to $T$, for every countable and up-directed subset $M$ of $T$. Clearly if $T$ is an up-directed and $\sigma$-closed subset of $\Sigma$ and $\mathbb{S}$ is a $\sigma$-complete inverse system indexed in $\Sigma$, then $\mathbb{S}|_T$ is also $\sigma$-complete.

\begin{proposition}\label{ft-semiopen-big}
Let $\mathbb{S}_1=\big((K_s)_{s \in \Sigma}, (p_{s}^{t})_{s \le t}\big)$ and $\mathbb{S}_2=\big((L_s)_{s \in \Sigma}, (q_{s}^{t})_{s \le t}\big)$ be $\sigma$-complete inverse systems of compact metric spaces. Let $\big(K, (p_s)_{s \in \Sigma}\big)$ be the inverse limit of $\mathbb{S}_1$ and $\big(L, (q_s)_{s \in \Sigma}\big)$ be the inverse limit of $\mathbb{S}_2$. If $f:K \to L$ is a continuous and semi-open map, then there exists a cofinal and $\sigma$-closed subset $T$ of $\Sigma$ such that for every $t \in T$, there exists a continuous and semi-open map $f_t:K_t \to L_t$ such that $q_t \circ f=f_t \circ p_t$.
\end{proposition}
\begin{proof}
Using \cite[Proposition~2.2]{KM}, we may assume without loss of generality that for every $s \in \Sigma$, there exists a continuous map $f_s: K_s \to L_s$ satisfying $q_s \circ f=f_s \circ p_s$. Define:
\[T=\{t \in \Sigma: \forall U \subset K_t \ \text{open, } U\neq\emptyset,\ \exists V \subset L_t \ \text{open, } V\neq\emptyset,\ \text{s.t.}\ q_t^{-1}[V] \subset f\big[p_t^{-1}[U]\big]\}.\]
It is easy to see that $f_t$ is semi-open, for every $t \in T$.
To see that $T$ is $\sigma$-closed in $\Sigma$, let $(t_n)_{n \in \omega}$ be an increasing sequence of elements of $T$ and $t=\sup_{n \in \omega} t_n \in \Sigma$. Since $\mathbb{S}_1$ is $\sigma$-complete, we have that $\big(K_t, (p_{t_n}^{t})_{n \in \omega}\big)$ is the inverse limit of $\mathbb{S}_1|_{\{t_n: n \in \omega\}}$ and therefore \cite[Proposition~2.5.5]{E} ensures that $\bigcup_{n \in \omega}\{(p_{t_n}^{t})^{-1}[W]: W \subset K_{t_n} \ \text{is open}\}$ is an open basis of $K_t$. Thus, fixed a nonempty open subset $U$ of $K_t$, there exist $n \in \omega$ and a nonempty open subset $W$ of $K_{t_n}$ such that $(p_{t_n}^{t})^{-1}[W] \subset U$. Since $t_n \in T$, there exists a nonempty open subset $V'$ of $L_{t_n}$ such that $q_{t_n}^{-1}[V'] \subset f \big[p_{t_n}^{-1}[W]\big]$. It is easy to see that if $V:=(q_{t_n}^{t})^{-1}[V']$, then $q_t^{-1}[V] \subset f\big[p_t^{-1}[U]\big]$. Now let us show that $T$ is cofinal in $\Sigma$. Fix $s \in \Sigma$ and set $S'=S \cup \{s, f, K, L, \Sigma, \varphi_K, \varphi_L, \psi_K, \psi_L\}$, where $S$ is the union of the countable sets from the statements of \cite[Lemma 7]{CCS} and \cite[Lemma 8]{CCS} and  $\varphi_K$, $\varphi_L$, $\psi_K$ and $\psi_L$ are the maps defined on $\Sigma$ given by $\varphi_K(t)=p_t$, $\varphi_L(t)=q_t$, $\psi_K(t)=\mathcal{B}_{K_t}$ and $\psi_L(t)=\mathcal{B}_{L_t}$, where $\mathcal{B}_{K_t}$ and $\mathcal{B}_{L_t}$ are fixed countable open basis of $K_t$ and $L_t$, respectively. Let $\Phi$ be the union of the finite lists of formulas from the statements of \cite[Lemma 7]{CCS} and \cite[Lemma 8]{CCS} enriched by the formulas (and their subformulas) marked as (*) in the proof below. According to Skolem's Theorem \cite[Theorem~4]{CCS} there exists a countable set $M$ such that $M \prec (\Phi;S')$. Note that \cite[Lemma 8 (1)]{CCS} implies that there exists $\delta=\sup(\Sigma \cap M) \in \Sigma$ and clearly $s \le \delta$. It follows from the $\sigma$-completeness of $\mathbb{S}_1$ and \cite[Proposition~2.5.5]{E} that to conclude that $\delta \in T$ it is enough to show that for every $t \in \Sigma \cap M$ and every $W \in \mathcal{B}_{K_t}$ nonempty, there exists a nonempty open subset $V$ of $L_\delta$ such that $q_\delta^{-1}[V] \subset f\big[p_{t}^{-1}[W]\big]$. Fix $t \in M \cap \Sigma$ and $W \in \mathcal{B}_{K_t}$ with $W \ne \emptyset$. Note that it follows from \cite[Lemma 7 (2)]{CCS} and \cite[Lemma 7 (4)]{CCS} that $W \in M$. Since $f$ is semi-open, if $U:=(p_{t}^{\delta})^{-1}[W]$, then there exists a nonempty open subset $W'$ of $L$ such that $W' \subset f\big[p_{\delta}^{-1}[U]\big]$ and thus $W' \subset f\big[p_t^{-1}[W]\big]$. We claim that we may assume that $W' \in M$. Indeed, consider the following formula
\[
\exists W' \subset L \ \text{open} \quad \big(W' \ne \emptyset \ \text{and} \ W' \subset f[p_t^{-1}[W]]\big).  \eqno{(*)}
\]
Since all free variables of this formula belong to $M$, by absoluteness we conclude that there exists $W' \in M$ satisfying $(*)$.
It follows from the fact that $\big(L, (q_s)_{s \in \Sigma}\big)$ is the inverse limit of $\mathbb{S}_2$ and \cite[Proposition~2.5.5]{E} that there exist $r \in \Sigma$ and $A \in \mathcal{B}_{L_r}$ such that $A \ne \emptyset$ and $q_r^{-1}[A] \subset W'$. Note that the absoluteness of the following formula (and its subformulas) for $M$ ensures that we may assume that $r \in M$
\[
\exists r \in \Sigma \quad \big(\exists A \in \mathcal{B}_{L_r} \ \text{s.t.}\ A \ne \emptyset \ \text{and} \ q_r^{-1}[A] \subset W'\big). \eqno{(*)}
\]
This concludes the proof, since $q_r^{-1}[A]=q_\delta^{-1}\big[(q_{r}^{ \delta})^{-1}[A]\big]$.
\end{proof}

\begin{remark}It is worth mentioning that Proposition \ref{ft-semiopen-big} is a version of \cite[Proposition~2.2]{KM} with semi-open maps in place of open ones. The definition of  $T$ in our proof is more technical than the one in \cite[Proposition~2.2]{KM}, since we have to circumvent one flaw present there. Indeed, the set $T$ defined there might fail to be $\sigma$-closed; we offer an instance of this phenomenon in Example \ref{T-not-complete} below. The statement of \cite[Proposition~2.2]{KM} is in any case correct and it can be proved with an argument similar to ours above.
\end{remark}

\begin{example}\label{T-not-complete}
For every $n \ge 1$, define $K_n=\{-1, -1/2, \ldots, -1/n,0, 1/n, \ldots, 1/2, 1\}$, endowed with the discrete topology and for every $n \le m$, define $p_{n}^{m}: K_m \to K_n$ as $p_{n}^{m}(x)=x$, if $x \in K_n$ and $p_{n}^{m}(x)=0$, otherwise. Set $K_\omega=\{-1/n: n \ge 1\} \cup \{0\} \cup \{1/n: n \ge 1\}$, endowed with the subspace topology of $\mathbb R$ and for each $n \in \omega$, let $p_{n}^{\omega}: K_\omega \to K_n$ be given by $p_{n}^{\omega}(x)=x$, if $x \in K_n$ and $p_{n}^{\omega}(x)=0$, otherwise. Finally, set $K_{\omega+1}=K_\omega$, let $p_{\omega}^{\omega+1}$ be the identity of $K_\omega$ and $p_{n}^{\omega+1} =p_{n}^{\omega}$, for every $n \in \omega$. It is easy to see that $\mathbb{S}_1=\big((K_{\alpha})_{\alpha \in [1, \omega+1]}, (p_{\alpha}^{\beta})_{\alpha \le \beta}\big)$ is a $\sigma$-complete inverse system whose inverse limit is $\big(K_{\omega+1}, (p_{\alpha}^{\omega+1})_{\alpha \in [1, \omega+1]}\big)$. Now for every $n \ge 1$, define $L_n=\{0, 1/n, \ldots, 1/2, 1\}$, endowed with the discrete topology and for every $n \le m$, define $q_{n}^{m}: L_m \to L_n$ as $q_{n}^{m}(x)=x$, if $x \in L_n$ and $q_{n}^{m}(x)=0$, otherwise. 
Set $L_\omega=\{0\} \cup \{1/n: n \ge 1\}$, endowed with the subspace topology of $\mathbb R$ and for each $n \in \omega$, define $q_{n}^{\omega}: L_\omega \to L_n$ as $q_{n}^{\omega}(x)=x$, if $x \in L_n$ and $q_{n}^{\omega}(x)=0$, otherwise. Set $L_{\omega+1}=K_{\omega+1}$ and let $q_{\omega}^{\omega+1}: L_{\omega+1} \to L_{\omega}$ be given by $q_{\omega}^{\omega+1}(x)=x$, if $x \in L_\omega$ and $q_{\omega}^{\omega+1}(x)=0$, otherwise. Finally, define $q_{n}^{\omega+1}= q_{n}^{\omega}\circ q_{\omega}^{\omega+1}$, for every $n \in \omega$. It is easy to see that $\mathbb{S}_2=\big((L_{\alpha}) _{\alpha \in [1, \omega+1]}, (q_{\alpha}^{\beta})_{\alpha \le \beta}\big)$ is a $\sigma$-complete inverse system whose inverse limit is $\big(L_{\omega+1}, (q_{\alpha}^{\omega+1})_{\alpha \in [1, \omega+1]}\big)$. Now let $f:K_{\omega+1} \to L_{\omega+1}$ be the identity of $K_{\omega+1}$ and, as in \cite[Proposition~2.2]{KM}, set:
\[T=\{\alpha \in [1,\omega+1]: \exists f_\alpha:K_\alpha \to L_\alpha \ \text{such that $f_\alpha$ is open and} \ q_{\alpha}^{\omega+1} \circ f=f_\alpha \circ p_{\alpha}^{\omega+1}\}.\]
Note that $\omega \subset T$, since for every $n \in \omega$, the open map $f_n: K_n \to L_n$ given by $f_n(x)=x$, if $x \in L_n$ and $f_n(x)=0$, otherwise, satisfies $q_{n}^{\omega+1} \circ f=f_{n} \circ p_{n}^{ \omega+1}$. However, $\omega=\sup_{n \in \omega} n$ does not belong to $T$, since $q_{\omega}^{\omega+1}$ is not open and this is the only map that could witness that $\omega \in T$. Therefore $T$ is not $\sigma$-closed in $[1, \omega+1]$.
\end{example}

\begin{corollary}\label{qt-semiopen}
Let $\mathbb{S}_1=\big((K_s)_{s \in \Sigma}, (p_{s}^{t})_{s \le t}\big)$ and $\mathbb{S}_2=\big((L_s)_{s \in \Sigma}, (q_{s}^{t})_{s \le t}\big)$ be $\sigma$-complete inverse systems of compact metric spaces. Let $\big(K, (p_s)_{s \in \Sigma}\big)$ be the inverse limit of $\mathbb{S}_1$, $\big(L, (q_s)_{s \in \Sigma}\big)$ be the inverse limit of $\mathbb{S}_2$ and $f:K \to L$ be a continuous, onto and semi-open map. If $p_{s}^{t}$ is semi-open, for every $s \le t$, then there exists a cofinal and $\sigma$-closed subset $T$ of $\Sigma$ such that $q_{s}^{t}$ is semi-open, for every $s, t \in T$ with $s \le t$.
\end{corollary}
\begin{proof}
Let $T$ be the cofinal and $\sigma$-complete subset of $\Sigma$ given by Proposition \ref{ft-semiopen-big}. 
Note that for every $t \in T$, it holds that $q_t \circ f=f_t \circ p_t$ is semi-open, since Lemma \ref{ps-semiopen-equiv-pst-semiopen} ensures that $p_t$ is semi-open and composition of semi-open maps is semi-open. Thus it follows from Lemma \ref{quotient-semiopen} applied to $q=f$, $h=q_t \circ f$ and $\bar{h}=q_t$ that $q_t$ is semi-open, for every $t \in T$. Finally, the result follows from Lemma \ref{ps-semiopen-equiv-pst-semiopen}, since $\big(L, (q_{t})_{t \in T}\big)$ is the inverse limit of $\mathbb{S}_2|_T$.
\end{proof}

\begin{theorem}\label{continuous-Valdivia-image-belongs-class}
Let $K$ be an element of $\mathcal{RS}$ and $L$ be a Valdivia compact space. If there exists a continuous, onto and semi-open map $f:K \to L$, then $L$ belongs to $\mathcal{RS}$.
\end{theorem}
\begin{proof}
Let $\mathbb{S}_1=\big((K_\alpha)_{\alpha \in \omega_1}, (p_{\alpha}^{\beta})_{\alpha \le \beta}\big)$ be a continuous inverse system of compact metric spaces such that $p_{\alpha}^{\beta}$ is a semi-open retraction, for every $\alpha \le \beta$ and let $\{p_\alpha:K \to K_\alpha, \alpha \in \omega_1\}$ be such that $\big(K, (p_\alpha)_{\alpha \in \omega_1}\big)$ is the inverse limit of $\mathbb{S}_1$. 
Note that $w(L) \le w(K) \le \omega_1$ and therefore, the only interesting case is when the weight of $L$ is $\omega_1$. It follows from \cite[Proposition 2.6]{KM} that there exist a continuous inverse system of compact metric spaces $\mathbb{S}_2=\big((L_\alpha)_{\alpha \in \omega_1}, (q_{\alpha}^{\beta})_{\alpha \le \beta}\big)$ such that each $q_{\alpha}^{\beta}$ is a retraction and a family $\{q_\alpha:L \to L_\alpha, \alpha \in \omega_1\}$ such that $\big(L, (q_\alpha)_{\alpha \in \omega_1}\big)$ is the inverse limit of $\mathbb{S}_2$. Let $T$ be the cofinal and $\sigma$-complete subset of $\omega_1$ given by Corollary \ref{qt-semiopen}. 
Since $L$ is the inverse limit of $\mathbb{S}_2|_T$, $T$ is order-isomorphic to $\omega_1$ and $\mathbb{S}_2|_{T}$ is continuous, we conclude that $L$ belongs to $\mathcal{RS}$.
\end{proof}

\begin{corollary}\label{closed-semiopen-retraction}
Let $K$ be an element of $\mathcal{RS}$, $L$ be a compact space and $f:K \to L$ be a continuous map.
\begin{enumerate}
    \item If $f$ is a semi-open retraction, then $L$ belongs to $\mathcal{RS}$;
    
    \item If $L$ is zero-dimensional and $f$ is an open surjection, then $L$ belongs to $\mathcal{RS}$.
\end{enumerate}
\end{corollary}
\begin{proof}
Note that \cite[Theorem 4.4]{KM} implies that $L$ is Valdivia in (1) and (2). Therefore the result follows from Theorem \ref{continuous-Valdivia-image-belongs-class}.
\end{proof}

%-----------------------------------------------------------%
%                                                           %
% 						BIBLIOGRAPHY 						%
%                                                           %
%-----------------------------------------------------------%

\end{document}